\documentclass[reqno,a4paper]{amsart}

\usepackage[latin1]{inputenc}
\usepackage[english]{babel}
\usepackage{eucal,amsfonts,amssymb,amsmath,amsthm,epsfig,mathrsfs}
\usepackage{cancel,soul}
\usepackage{color}
\usepackage{amscd,amsxtra}
\usepackage{enumerate}
\usepackage{latexsym}
\usepackage{bm}

\allowdisplaybreaks

\newcounter{ipotesi}

\Alph{ipotesi}

\newtheorem{thm}{Theorem}[section]
{\rm}
{\rm}

\newtheorem{cor}[thm]{Corollary}
\newtheorem{pro}[thm]{Proposition}

\newtheorem{rmk}[thm]{Remark}{\rm}

\newcounter{parentenv}

\newcommand{\R}{{\mathbb R}}

\newcommand{\F}{\mathcal F}

\newcommand{\ra}{\rightarrow}

\newcommand{\set}[1]{{\left\{#1\right\}}}

\newcommand{\gen}[1]{{\left\langle #1\right\rangle}}
\newcommand{\abs}[1]{{\left|#1\right|}}
\newcommand{\norm}[1]{{\left\|#1\right\|}}

\newcommand{\eqsys}[1]{{\left\{\begin{array}{ll}#1\end{array}\right.}}
\newcommand{\tc}{\, \middle |\,}

                                              \newcommand{\parti}{\operatorname{\mathcal{P}}}
\newcommand{\con}{\operatorname{\mathscr{C}}}

\newcommand{\weak}{\operatorname{\mathnormal{w}}}

\newcommand{\borel}{{\operatorname{Bo}}}

\newcommand{\continuum}{\operatorname{\mathfrak{c}}}
\newcommand{\baire}{{\operatorname{Ba}}}

\makeindex

\begin{document}

\frenchspacing

\title[Weak-star and norm Borel sets in the dual of the space of continuous functions]{A note on weak-star and norm Borel sets in the dual of the space of continuous functions}

\author[S. Ferrari]{{S. Ferrari}}

\address[S. Ferrari]{Dipartimento di Scienze Matematiche, Fisiche e Informatiche, Universit\`a degli Studi di Parma, Parco Area delle Scienze 53/A, 43124 Parma, Italy.}
\email{\textcolor[rgb]{0.00,0.00,0.84}{simone.ferrari1@unipr.it}}


\begin{abstract}
Let $\borel(T,\tau)$ be the Borel $\sigma$-algebra generated by the topology $\tau$ on $T$. In this paper we show that if $K$ is a Hausdorff compact space, then every subset of $K$ is a Borel set if, and only if, 
\[\borel(C^*(K),\weak^*)=\borel(C^*(K),\norm{\cdot});\]
where $\weak^*$ denotes the weak-star topology and $\norm{\cdot}$ is the dual norm with respect to the sup-norm on the space of real-valued continuous functions $C(K)$. Furthermore we study the topological properties of the Hausdorff compact spaces $K$ such that every subset is a Borel set. In particular we show that, if the axiom of choice holds true, then $K$ is scattered.
\end{abstract}

\date{\today}

\maketitle

\section{Introduction}

Due to the presence of several important topologies on a Banach space it is natural to ask if there is any relationship between the Borel $\sigma$-algebras generated by these different topologies. Many authors have studied the relationship between the Borel sets generated by the weak topology and the one generated by the norm topology (see for example \cite{Edgar1,Edgar2,Oncina,Raja2}) and have found various conditions for the coincidence of this two classes. In particular it is shown that the coincidence of the above $\sigma$-algebras is related to the existence of special types of equivalent norms (see \cite{FOOR16,FOR16,FOR19,MOTV09,Oncina,Raja2}, for a study of these types of equivalent renormings).


The subject of this paper is to understand the topology of a compact space $K$ such that the weak-star and norm Borel structure of $C^*(K)$ agree.

\section{Notations and preliminaries}

A family $\mathcal{N}$ of subsets of a topological space $(T,\tau)$ is a \emph{network} for $T$ if for every point  $t\in T$ and every neighbourhood $U$ of $t$ there exists $N\in\mathcal{N}$ such that $t\in N\subseteq U$ (see \cite{Engelking} for more informations). We remark that the definition of network differs form the definition of basis of a topological space, indeed it is not required for the sets of a network to be open.

We will denote by $P(X)$ the power set of a set $X$.

Let $\tau$ be a topology on the set $T$, then we will use the symbol $\borel(T,\tau)$ to denote the Borel $\sigma$-algebra generated by the topology $\tau$ on $T$, while we will denote with $\baire(T,\tau)$ the Baire $\sigma$-algebra, i.e. the smallest $\sigma$-algebra with respect to which all the $\tau$-continuous real-valued functions are measurable. Finally the symbol $\baire_p(T,\tau)$ will denote the $\sigma$-algebra of the sets with the Baire property in $(T,\tau)$, i.e. all sets of the form $U\Delta M$, where $U$ is open and $M$ is of first category (see \cite[Theorem 4.1]{Oxtoby}).

Given two topological spaces $(X,\tau_X)$ and $(Y,\tau_Y)$, then a map $F:X\ra P(Y)$ is said to be upper semicontinuous at a point $x\in X$ (\emph{usc at $x$}, for short) if, for every open set $V$ containing $F(x)$, there exists a neighbourhood $U$ of $x$ such that
\[F(U)=\bigcup\set{F(u)\tc u\in U}\subseteq V.\]
We say that $F$ is upper semicontinuous (\emph{usc}, for short) if it is upper semicontinuous at $x$ for every point $x\in X$. We say that a map $F$ is \emph{usco} if it is usc and takes non-empty compact values.

We will denote by $\aleph_0$ and $\continuum$ the cardinality of the set of the natural numbers and of its power set, respectively. We will use the symbol $\omega_1$ to denote the first uncountable ordinal.

\section{The main results}

Throughout this section $(K,\tau)$ will denote a Hausdorff compact space and $C(K)$ the space of real-valued continuous functions on $K$ endowed with the sup-norm. In the main theorem we prove that, if $\borel(C^*(K),\weak^*)=\borel(C^*(K),\weak)$, then 
\[\borel(C^*(K),\weak^*)=\borel(C^*(K),\norm{\cdot}).\]

\begin{pro}
If $K$ contains a non-Borel subset, then $\borel(C^*(K),\weak^*)\neq\borel(C^*(K),\weak)$.
\end{pro}

\begin{proof}
It is a well known fact that $(K,\tau)$ is homeomorphic to $( K_0 ,\weak^*)$, where $ K_0 =\set{\delta_k\tc k\in K}$, the set of the Dirac measures concentrated in $k\in K$. So $ K_0 $ is a compact subset of $(C^*(K),\weak^*)$ and, in particular, it is closed with respect to the weak topology.

We claim that $ K_0 $ is discrete with respect to the weak topology. Indeed consider the family of functions $\mathcal{A}=\set{f_k:K\ra\R\tc k\in K}$ defined as follows
\[f_k(x)=\eqsys{1& x=k,\\ 0& x\neq k,}\qquad k,x\in K.\]
For every $k\in K$, $f_k$ is a Borel function and can be seen as an element of $C^{**}(K)$ in the following way:
\[_{C^*(K)}\gen{\mu,f_k}_{C^{**}(K)}=\int_Kf_kd\mu,\qquad\mu\in C^*(K).\]
Observe that
\[\set{\mu\in  K_0 \tc f_k(\mu)>\frac{1}{2}}=\set{\delta_k}.\]
So $ K_0 $ is discrete and closed with respect to the weak topology, which implies that $\borel( K_0 ,\weak)=P( K_0 )$. Finally if it holds that $\borel(C^*(K),\weak^*)=\borel(C^*(K),\weak)$, then $\borel( K_0 ,\weak^*)=\borel( K_0 ,\weak)=P( K_0 )$, a contradition.
\end{proof}

We are now interested in the topological properties of compact spaces $K$ such that every subset of $K$ is a Borel set. As one may expect this properties are strictly related to some set-theoretic axioms.

\begin{pro}
Let the continuum hypothesis hold true and let $(X,\tau)$ be a Hausdorff space with a countable network. If every subset of $X$ is a Borel set, then $X$ is finite or countable.
\end{pro}

\begin{proof}
By \cite[paragraph 4A3F]{Fremlin}, $\abs{\borel(X,\tau)}\leq\continuum$. 
We have
\[2^{\abs{X}}=\abs{\parti(X)}=\abs{\borel(X,\tau)}\leq\continuum.\]
By the continuum hypothesis follows that $\abs{X}\leq\aleph_0$, since $\continuum=2^{\aleph_0}$.
\end{proof}

\begin{thm}\label{teorema mio}
Let the axiom of choice holds true and let $(K,\tau)$ be a Hausdorff compact space. If every subset of $K$ is a Borel set, then $K$ is scattered.
\end{thm}

\begin{proof}
By \cite{Pelczynski_Semadeni} a Hausdorff compact space $K$ is scattered if, and only if, $[0,1]$ is not a continuous image of $K$. By contradition let $f:K\ra[0,1]$ be a continuous surjection and
consider the multifunction
\[F=f^{-1}:[0,1]\longrightarrow \mathcal{K}^*(K),\]
where $\mathcal{K}^*(K)$ is the colletion of of all compact non-empty subset of $K$.  $F$ is a compact and non-empty valued multifunction, and recalling that for every $B\in P([0,1])$
\[F^{-1}(B):=\set{x\in [0,1]\tc F(x)\cap B\neq\emptyset}=f(B)\]
and that continuous function from a compact space to an Hausdorff space are closed (see \cite[pag. 169]{Engelking}), we obtain that $F$ is an usco map. Since $[0,1]$ is a compact second countable space, in
particular a Baire space, and $K$ is a completely regular space (see \cite[pag. 196]{Engelking}) we can apply \cite[Theorem 6]{Graf} and get
\[f_0:[0,1]\longrightarrow K\]
a $\baire_p([0,1],\tau_{\R})$--$\borel(K,\tau)$-selection of $F$.

We claim that $f_0$ is injective. Indeed if $x,y\in [0,1]$ and $f_0(x)=f_0(y)$, then $f_0(x)\in f^{-1}(x)$ and $f_0(y)\in f^{-1}(y)$. But we know that $f^{-1}(x)\cap f^{-1}(y)=\emptyset$, so $x=y$. Recalling that for injective function it holds 
\[f_0^{-1}(f_0(A))=A\qquad\text{for every } A\in P([0,1]),\]
and every subset of $K$ is a Borel set we have $\baire_p([0,1],\tau_\R)=P([0,1])$, which is a contradiction by \cite[Theorem 5.5]{Oxtoby}.
\end{proof}

One may think that a space with only Borel subsets should be meager, but by \cite{Shelah} the existence of a measurable cardinal is equiconsistent with the existence of a non-meager T$_4$ space with no isolated point in which every subset is the union of an open and a closed set.

\begin{cor}\label{corollary}
If every subset of $K$ is a Borel set, then
\[\borel(\con^*(K),\weak)=\borel(\con^*(K),\norm{\cdot}).\]
Furthermore if $\borel(\con^*(K),\weak^*)=\borel(\con^*(K),\weak)$, then $\borel(\con^*(K),\weak^*)=\borel(\con^*(K),\norm{\cdot})$.
\end{cor}

\begin{proof}
By Theorem \ref{teorema mio} $K$ is scattered and by \cite[Lemma 8.3 of Chapter VI]{Deville} $\con(K)$ is an Asplund space. So by \cite{Fabian_Godefroy} $\con^*(K)$ admits a LUR renorming, and using \cite[Corollary 2.4]{Edgar2} one obtain the thesis.
\end{proof}

We want to stress that the furthermore part of Corollary \ref{corollary} is not typical for a generic dual space. Indeed by \cite[Proposition 8]{Edgar80} the dual of the James space $J(\omega_1)$ has a weak Borel set which is not a weak-star Borel set, but since the dual of $J(\omega_1)$ is a Asplund space then, by \cite[Lemma 8.3 of Chapter VI]{Deville} and \cite[Corollary 2.4]{Edgar2}, it holds $\borel(J^*(\omega_1),\weak)=\borel(J^*(\omega_1),\norm{\cdot})$.

\begin{rmk}
By \cite[Corollary 4.4]{Raja} $\con^*(K)$ admits a $\weak^*$-Kadets norm if, and only if, $K$ is a countable union of relatively discrete subsets, so every subset of $K$ is a Borel set. But if we assume the Martin axiom and the negation of the continuum hypothesis, then there exists an uncountable subset $X$ of $\R$ such that every subset of $X$ is a relative $\F_\sigma$, see \cite{Martin_Solovay}. In particular if we consider the one-point compatification $\alpha(X)$ of $X$, then every discrete subset of $\alpha(X)$ is countable. So $\alpha(X)$ is not the countable union of relative discrete subsets.
\end{rmk}

\bibliographystyle{acm}
\nocite{*} \addcontentsline{toc}{chapter}{Bibliografia}
\bibliography{bibborstrc}
\markboth{\textsc{References}}{\textsc{References}}

\end{document}